%% file: paper.tex
\documentclass[11pt,a4paper]{amsart}
\usepackage{hyperref,amsthm,amsmath,amssymb,graphicx}
\usepackage[numbers,sort&compress]{natbib}

\newcommand{\spacing}[1]{\renewcommand{\baselinestretch}{#1}\setlength{\footnotesep}{\baselinestretch\footnotesep}}
\spacing{1.2}  

\input{commands.tex}   

\newcommand{\ff}{First-Fit chain partition}
\newcommand{\FF}{\mathrm{FF}}
\newcommand{\val}{\mathrm{val}}

\begin{document}

\title[First-Fit is Linear on Posets Excl. Long Incomparable Chains]{First-Fit is Linear on Posets Excluding Two Long Incomparable Chains}
\date{\today}

\author{Gwena\"el Joret}
\address{D\'epartement d'Informatique\\
Universit\'e Libre de Bruxelles\\
Brussels, Belgium}
\email[Gwena\"el Joret]{gjoret@ulb.ac.be}
\thanks{
This work was supported in part by the Actions de Recherche 
Concert\'ees (ARC) fund of the Communaut\'e fran\c{c}aise de Belgique. 
The first author is a Postdoctoral Researcher of the 
Fonds National de la Recherche Scientifique (F.R.S.--FNRS)}

\author{Kevin G. Milans}
\address{Department of Mathematics\\
University of South Carolina\\
Columbia, South Carolina}
\email[Kevin G. Milans]{milans@math.sc.edu}
\thanks{
The second author acknowledges support of the National Science Foundation through a fellowship funded by the grant ``EMSW21-MCTP: Research Experience for Graduate Students'' (NSF DMS 08-38434).
}

\sloppy
\maketitle

\begin{abstract}
A poset is $(\chain{r}+\chain{s})$-free if it does not contain 
two incomparable chains of size $r$ and $s$, respectively.
We prove that when $r$ and $s$ are at least $2$, the First-Fit algorithm 
partitions every $(\chain{r}+\chain{s})$-free poset $P$ into at most
$8(r-1)(s-1)w$ chains, where $w$ is the width of $P$.
This solves an open problem of Bosek, Krawczyk, and Szczypka 
({\em SIAM J. Discrete Math., 23(4):1992--1999, 2010}).
\end{abstract}

\section{Introduction}

A {\DEF chain} in a poset is a set of elements that are pairwise comparable, and an {\DEF antichain} is a set of elements that are pairwise incomparable.  The {\DEF height} of a poset is the size of a largest chain, and the {\DEF width} is the size of a largest antichain.  In the {\DEF on-line chain partitioning problem}, the elements of an unknown poset $P$ are revealed one by one in some order.  Each time a new element $x$ is presented, one has to assign a color to $x$, maintaining the property that each color class is a chain. The goal is to minimize the number of chains in the resulting chain partition of $P$.

This classical problem has received increased attention in the recent years; see, for 
example, the survey by Bosek, Felsner, Kloch, Krawczyk, Matecki, and Micek~\cite{BFKKMM_sub}. 
In this context, the quality of a solution is typically compared against the width $w$ of $P$.  Since elements of an antichain must receive distinct colors, at least $w$ colors are needed.  By Dilworth's theorem, if all elements of $P$ are presented before any are colored, then $w$ colors suffice.  In the on-line setting, more colors are needed.

Let $\val(w)$ be the least $k$ such that there is an on-line algorithm that partitions posets of width $w$ into at most $k$ chains.  Establishing that $\val(w)$ is finite when $w\ge 2$ is challenging.  In 1981, Kierstead~\cite{K_CM} proved that $\val(w) \le (5^w - 1)/4$.  For nearly three decades, Kierstead's result was the best known upper bound on $\val(w)$.  Recently, Bosek and Krawczyk~\cite{BK_prepa} showed that $\val(w) \le w^{16 \lg w}$ (see~\cite{BFKKMM_sub} for a proof sketch).  From below, Szemer\'edi proved that $\val(w) \ge \binom{w+1}{2}$ (see~\cite{K_CM, BFKKMM_sub}), and Bosek {\it et al.}~\cite{BFKKMM_sub} showed that $\val(w) \ge (2-o(1))\binom{w+1}{2}$.  One of the central questions in the theory of on-line problems on partial orders is whether $\val(w)$ is bounded above by a polynomial in $w$.

In this paper, we are interested in the performance of an on-line chain partitioning algorithm called First-Fit.  Using the positive integers for colors, First-Fit colors $x$ with the least $j$ such that $x$ and all elements previously assigned color $j$ form a chain.  It is known that, for general posets, the number of chains used by First-Fit is not bounded by a function of $w$. In fact, Kierstead~\cite{K_CM} showed that First-Fit uses arbitrarily many chains on posets of width $2$ (see also~\cite{BKS_SIDMA}).

Nevertheless, First-Fit performs well on certain classes of posets, such as interval orders.  An \emph{interval order} is a poset whose elements are closed intervals on the real line, with $[a,b] < [c,d]$ if and only if $b<c$.  Let $\FF(w)$ be the maximum number of chains that First-Fit uses on interval orders of width $w$.  Kierstead~\cite{K_SIDMA} proved that $\FF(w)\le 40w$.  Kierstead and Qin~\cite{KQ_DM} subsequently improved the bound, showing that $\FF(w)\le 25.8w$.  Later, Pemmaraju, Raman, and Varadarajan~\cite{PRV_SODA} (see also \cite{PRV_TALG}) proved that $\FF(w)\le 10w$ with an elegant argument known as the Column Construction Method.  Their proof was later refined by Brightwell, Kierstead, and Trotter~\cite{BKT_mscript} and independently by Narayanaswamy and Babu~\cite{NS_ORDER} to show that $\FF(w)\le 8w$.  

From early results of Kierstead and Trotter~\cite{KT_SCCGTC}, it follows that $\FF(w) \ge (3+\eps)w$ for some positive $\eps$.  Chrobak and \'Slusarek~\cite{CS_RAIRO} showed that $\FF(w) \ge 4w-9$ when $w\ge 4$ and subsequently improved the multiplicative constant to $4.45$ at the expense of a weaker additive constant.  In 2004, Kierstead and Trotter~\cite{KT_unpub} proved that $\FF(w) \ge 4.99w - c$ for some constant $c$ with the aid of a computer.  Recently, Kierstead, Smith, and Trotter~\cite{KST_unpub} proved that for each positive $\eps$, there is a constant $c$ such that $\FF(w) \ge (5-\eps)w - c$.

If $P$ and $Q$ are posets, then $P+Q$ denotes the poset obtained from disjoint copies of $P$ and $Q$ where elements in the copy of $P$ are incomparable to elements in the copy of $Q$.  A poset $P$ is {\DEF $Q$-free} if no induced subposet of $P$ is isomorphic to $Q$.  We denote by {\DEF $\chain{r}$} the poset consisting of a chain of size $r$.  Fishburn~\cite{F_JMP} characterized the interval orders as the posets that are $(\chain{2}+\chain{2})$-free.  When $r$ and $s$ are at least two, the family of $(\chain{r}+\chain{s})$-free posets contains the family of interval orders.  Bosek, Krawczyk, and Szczypka~\cite{BKS_SIDMA} showed that when $r\ge s$, First-Fit partitions every $(\chain{r}+\chain{s})$-free poset into at most $(3r - 2)(w - 1)w + w$ chains.  They asked whether First-Fit uses only a linear number of chains, in terms of $w$, on $(\chain{r}+\chain{s})$-free posets, as it does on interval orders. This question also appears in the survey of Bosek {\it et al.}~\cite{BFKKMM_sub} and in a recent paper of Felsner, Krawczyk, and Trotter~\cite{FKT_sub}.

We give a positive answer to this question by showing that First-Fit partitions every $(\chain{r}+\chain{s})$-free poset into at most $8(r-1)(s-1)w$ chains.  As far as we know, this also provides the first proof that some on-line algorithm uses $o(w^{2})$ chains on $(\chain{r}+\chain{s})$-free posets.  Our proof is strongly influenced by the Column Construction Method of Pemmaraju {\it et al.}~\cite{PRV_TALG} and can be viewed as a generalization of that technique from interval orders to $(\chain{r}+\chain{s})$-free posets.

In Section~\ref{sec_evo}, we present our generalization of the Column Construction Method and establish several of its properties.  In Section~\ref{sec_thm}, we combine these results with a structural lemma about $(\chain{r}+\chain{s})$-free posets to obtain our main result.

\section{Evolution of Societies}
\label{sec_evo}

Let $P$ be a poset.  A {\DEF \ff} is an ordered partition $C_{1}, \dots, C_{m}$ of $P$ into non-empty chains such that if $i<j$ and $x\in C_j$, then some element in $C_i$ is incomparable to $x$.  Note that if $C_1, \ldots, C_m$ is a \ff , then First-Fit produces this partition when elements in $C_1$ are presented first, followed by elements in $C_2$, and continuing through elements in $C_m$.
Conversely, every ordered partition produced by First-Fit is a \ff .

A {\DEF group} is a set of elements in $P$.  A {\DEF $t$-society} is a pair $(S,F)$ where $S$ is a set of groups and $F$ is a {\DEF friendship function} from $S\times [t]$ to $S\cup\{\star\}$, where $[t]$ denotes the set $\{1,\ldots,t\}$.  Each group $X\in S$ has slots for up to $t$ friends.  We say that {\DEF $X$ lists $Y$ as a friend in slot $k$} if $F(X,k)=Y$.  It is possible that $X$ does not list any friend in slot $k$, in which case $F(X,k)=\star$.

The overview of our proof is as follows.  Given an \rsfree\ poset $P$, we first exploit the structure of $P$ to define an initial $t$-society $(S_0,F_0)$ for some $t$ depending on $s$.  Next, we fix a \ff\ $C_1, \ldots, C_m$, which we extend to an infinite sequence of chains by defining $C_j = \ns$ for $j>m$.  We allow the initial $t$-society to evolve, generating a sequence of $t$-societies $(S_0,F_0),\ldots,(S_n,F_n)$.  For $j\ge 1$, the $t$-society $(S_j,F_j)$ is obtained from $(S_{j-1},F_{j-1})$ by following certain rules that depend on $C_j$ and the previous transitions.  It is helpful to view the $t$-societies as vertices of a path and to associate the edge joining $(S_{j-1},F_{j-1})$ and $(S_j,F_j)$ with the chain $C_j$.

During the evolution, we maintain that $S_0 \supseteq S_1 \supseteq \cdots \supseteq S_n$.  The evolution ends when a $t$-society $(S_n,F_n)$ is generated where $S_n = \ns$.  The proof proceeds in two parts.  First, we show that a long evolution implies that some group in the initial $t$-society is large.  Second, given an \rsfree\ poset $P$, we show how to construct an initial $t$-society of groups inducing subposets of height at most $r-1$ that leads to a long evolution.  Because large posets of bounded height contain large antichains, we obtain a lower bound on the width of $P$.

In our societies, friendship is a lifetime commitment: if $F_{j-1}(X,k)=Y$ and $\{X,Y\}\subseteq S_j$, then $F_j(X,k)=Y$.  If $X$ survives the transition from $S_{j-1}$ to $S_j$ but $Y$ does not, then $X$ either chooses a new friend for its $k$th slot or leaves its $k$th slot empty according to the rules of a {\DEF replacement scheme}.  We postpone the presentation of the details of our replacement scheme and the construction of the initial $t$-society.

A group $X$ may survive the transition from $S_{j-1}$ to $S_j$ in three ways, each of which defines a transition type.  We use the first three Greek letters $\alpha$, $\beta$, and $\gamma$ to name the transition types.  When $a \in \{\alpha,\beta,\gamma\}$ and $i\le j$, we define $N_{i,j}^a(X)$ to be the number of transitions of type $a$ that $X$ makes in the evolution from $(S_i,F_i)$ to $(S_j,F_j)$.

Let $\eps = 1/2t$; in Lemma~\ref{lem:large-group}, we will find that this choice of $\eps$ is optimal.  We now describe the rules that govern which groups survive the $j$th transition from $S_{j-1}$ to $S_j$.  Let $X$ be a group in $S_{j-1}$.
\begin{enumerate}
\item If $X$ has non-empty intersection with $C_j$, then $X$ makes an $\alpha$-transition from $S_{j-1}$ to $S_j$.
\item Otherwise, if some friend of $X$ in the $t$-society $(S_{j-1},F_{j-1})$ has non-empty intersection with $C_j$, then $X$ makes a $\beta$-transition from $S_{j-1}$ to $S_j$.
\item Otherwise, if there is an $i$ such that $N_{i,j-1}^\alpha(X) > \eps(j-i)$, then $X$ makes a $\gamma$-transition from $S_{j-1}$ to $S_j$.
\end{enumerate}
If none of the three rules apply, then $X\not\in S_j$, and other groups that list $X$ as a friend and survive to $S_j$ update their list of friends according to the replacement scheme.

First, we show that a long evolution implies that some group is large.  We need several lemmas.

\begin{lemma}\label{lem:single-col}
Fix an evolution $(S_0,F_0),\ldots,(S_n,F_n)$ of $t$-societies.  Let $Y_1, Y_2, \ldots, Y_q$ be a list of groups and let  $[a_1,b_1], \ldots, [a_q,b_q]$ be a sequence of disjoint intervals with integral endpoints in $[0,n]$ such that $b_j$ is the largest integer such that $Y_j \in S_{b_j}$.  The sum $\sum_{j=1}^q N_{a_j,b_j}^\alpha(Y_j)$ is at most $\eps n$. 
\end{lemma}
\begin{proof}
If $a_j=b_j$, then clearly $N_{a_j,b_j}^\alpha(Y_j) = 0$.  Hence, we may assume that $0 \le a_1<b_1<\cdots <a_q <b_q$.  Also, $b_q < n$ because $S_n = \ns$.  Note that $Y_j \not\in S_{b_j+1}$.  It follows that $Y_j$ did not satisfy the third condition in the transition from $S_{b_j}$ to $S_{b_j + 1}$ and therefore $N_{a_j,b_j}^\alpha (Y_j) \le \eps(b_j + 1 - a_j)$.  Also, $\sum_{j=1}^q (b_j+1 - a_j) \le n$ because the intervals $[a_j,b_j]$ are disjoint subsets of $[0,n-1]$ with integral endpoints.
\end{proof}

Our next lemma provides a bound on the number of $\beta$-transitions that a group can make if it survives to the last non-empty $t$-society.

\begin{lemma}\label{lem:few-betas}
Fix an evolution $(S_0,F_0),\ldots,(S_n,F_n)$ of $t$-societies.  If $X \in S_{n-1}$, then $N_{0,n-1}^\beta(X) \le t\eps n$.
\end{lemma}
\begin{proof}
Let $X \in S_{n-1}$, and for each $k\in[t]$, let $\mc{Y}_k$ be the set of groups that $X$ lists as a friend in slot $k$ at some point in the evolution.  If $X$ makes a $\beta$-transition from $S_{j-1}$ to $S_j$, then there is a slot $k$ and group $Y\in \mc{Y}_k$ such that $F_{j-1}(X,k) = Y$ and $Y$ has non-empty intersection with $C_j$.  Because $Y \in S_{j-1}$ and $Y$ has non-empty intersection with $C_j$, we have  that $Y$ makes an $\alpha$-transition from $S_{j-1}$ to $S_j$.  It follows that
\[
	N_{0,n-1}^\beta(X) \le \sum_{k=1}^t \sum_{Y\in\mc{Y}_k} N_{I(Y)}^\alpha(Y)
\]
where $I(Y)$ is denotes the interval during which $X$ lists $Y$ as a friend.  (Formally, $j\in I(Y)$ if and only if $F_j(X,k)=Y$ for some $k\in[t]$.)  It suffices to show that $\sum_{Y\in\mc{Y}_k} N_{I(Y)}^\alpha(Y) \le \eps n$ for each $k\in[t]$.  Because $\{I(Y)\st Y\in\mc{Y}_k\}$ are disjoint intervals, the bound follows from Lemma~\ref{lem:single-col}.
\end{proof}

Next, we show that for each group $X$, the $\alpha$-transitions that $X$ makes constitute a large fraction of the total number of $X$'s transitions not of type $\beta$.

\begin{lemma}\label{lem:many-alphas}
Fix an evolution $(S_0,F_0),\ldots,(S_n,F_n)$ of $t$-societies.  If $X$ is a group, then $N_{0,j}^\alpha(X) \ge \eps(N_{0,j}^\alpha(X) + N_{0,j}^\gamma(X))$ for each $j$ with $X\in S_j$.
\end{lemma}
\begin{proof}
If $j=0$, then the inequality holds.  For $j\ge 1$, the inequality holds immediately by induction unless $X$ makes a $\gamma$-transition from $S_{j-1}$ to $S_j$.  In this case, there is some $i$ such that $N_{i,j-1}^\alpha(X) > \eps (j-i)$.  Applying the inductive hypothesis to obtain a lower bound on $N_{0,i}^\alpha(X)$, it follows that 
\begin{align*}
N_{0,j}^\alpha(X) &= N_{0,i}^\alpha(X) + N_{i,j-1}^\alpha(X) \\
	& \ge \eps(N_{0,i}^\alpha(X)+N_{0,i}^\gamma(X)) + \eps(j-i) \\
	& \ge \eps(N_{0,i}^\alpha(X)+N_{0,i}^\gamma(X)) + \eps(N_{i,j}^\alpha(X)+N_{i,j}^\gamma(X)) \\
	& = \eps(N_{0,j}^\alpha(X)+N_{0,j}^\gamma(X))
\end{align*}
as required.
\end{proof}

We are now able to show that a long evolution implies that some group is large.

\begin{lemma}\label{lem:large-group}
Fix an evolution $(S_0,F_0),\ldots,(S_n,F_n)$ of $t$-societies.  If $X\in S_{n-1}$, then $|X| \ge (n-2)/4t$.
\end{lemma}
\begin{proof}
Whenever $X$ makes an $\alpha$-transition from $S_{j-1}$ to $S_j$, it has non-empty intersection with chain $C_j$.  Because the chains are disjoint, it follows that $|X|\ge N_{0,n-1}^\alpha(X)$.  By Lemma~\ref{lem:many-alphas}, we have that $N_{0,n-1}^\alpha(X) \ge \eps(N_{0,n-1}^\alpha(X) + N_{0,n-1}^\gamma(X))$.  Note that $X$ makes $n-1$ transitions in total, because $X\in S_{n-1}$.  Hence $N_{0,n-1}^\alpha(X) + N_{0,n-1}^\beta(X) + N_{0,n-1}^\gamma(X) = n-1$.  By Lemma~\ref{lem:few-betas}, we have that $N_{0,n-1}^\alpha(X) + t\eps n + N_{0,n-1}^\gamma(X) \ge n-1$.  Consequently, $N_{0,n-1}^\alpha(X) \ge \eps(n-1 - t \eps n)$.  With $\eps = 1/2t$, we obtain $N_{0,n-1}^\alpha(X) \ge (n-2)/4t$ as required.
\end{proof}

\section{The Initial Society and Replacement Scheme}
\label{sec_thm}

It remains to describe the initial $t$-society and our replacement scheme.  Both depend on the following structural lemma about \rsfree\ posets.  The {\DEF height} of an element $x$, denoted $h(x)$, is the size of a largest chain with maximum element $x$.

\begin{lemma}\label{lem:init-society}
Let $r$ and $s$ be integers with $r\ge 2$ and $s\ge 2$, and let $P$ be an \rsfree\ poset.  There is a function $I$ which assigns to each element $x\in P$ a non-empty set of consecutive integers $I(x)$ with the following properties.
\begin{enumerate}
	\item For each integer $k$, the set $\{x\in P\st k\in I(x)\}$ induces a subposet of height at most $r-1$.
	\item If $x$ and $y$ are incomparable in $P$, then either $I(x)$ and $I(y)$ have non-empty intersection, or at most $s-2$ integers are strictly between $I(x)$ and $I(y)$.
\end{enumerate}
\end{lemma}
\begin{proof}
Let $q$ be the height of $P$.  For each $x\in P$, let $Z(x)$ be the set of all elements $z$ such that $P$ contains a chain of size $r$ with minimum element $x$ and maximum element $z$.  When $Z(x)$ is non-empty, define $b(x)$ to be the minimum height of an element in $Z(x)$; we set $b(x)=q+1$ when $Z(x)=\ns$.  Let $I(x) = \{h(x), \ldots, b(x) - 1\}$.
	
Fix an integer $k$ and let $X = \{x\in P\st k\in I(x)\}$.  Suppose for a contradiction that $X$ contains a chain $x_1 < \cdots < x_r$.  Since $x_r \in X$, we have that $k \in I(x_r)$, which implies that $h(x_r) \le k$.  Similarly, $k\in I(x_1)$ and therefore $k \le b(x_1) - 1$.  Since $x_r \in Z(x_1)$, it follows that $b(x_1)\le h(x_r)$.  Hence $h(x_r)\le k \le h(x_r) - 1$, a contradiction.  It follows that (1) holds.

It remains to check (2).  Suppose that $x$ and $y$ are incomparable.  If $I(x)$ and $I(y)$ have non-empty intersection, then (2) holds.  Hence, we may assume that every integer in $I(x)$ is less than every integer in $I(y)$.  Let $i$ be the greatest integer in $I(x)$ and let $j$ be the least integer in $I(y)$, and note that $i<j\le q$.  Since $i\in I(x)$ but $i+1\not\in I(x)$, it follows that $b(x)-1 = i$.  Because $i<q$, it follows that $b(x)=i+1\le q$ and therefore $Z(x) \ne\ns$.  Hence, there is a chain $x=x_1<\cdots <x_r$ in $P$ with $h(x_r) = i+1$.  Similarly, $h(y) = j$ and there is a chain $y = y_j > \cdots > y_1$ in $P$ with $h(y_k)=k$ for each $k\in [j]$.  Let $X=\{x_1,\ldots,x_r\}$ and let $Y=\{y_{i+1},\ldots,y_j\}$.  We claim that every element in $X$ is incomparable to every element in $Y$.  If $x_a \le y_b$, then transitivity implies that $x = x_1 \le y_j = y$, contrary to the assumption that $x$ and $y$ are incomparable. Conversely, if $y_a \le x_b$, then transitivity implies that $y_{i+1}\le x_r$.  But $y_{i+1}\le x_r$ is impossible because $y_{i+1}$ and $x_r$ are distinct (since $x_{r} \not \leq y_{i+1}$) and have the same height.  Hence every element in $X$ is incomparable to every element in $Y$ as claimed.

It follows that $X \cup Y$ induces a copy of $\chain{r} + \chain{j-i}$ in $P$.  Because $P$ is \rsfree , we have that $j-i \le s-1$ and therefore the set of integers $\{i+1,\ldots,j-1\}$ strictly between $I(x)$ and $I(y)$ has size at most $s-2$.
\end{proof}

We now have the tools necessary to describe the initial $t$-society and our replacement scheme.  While our transition rules require only that each $S_j$ is a set of groups, our replacement scheme imposes additional structure on $S_j$.  In particular, our replacement scheme treats $S_j$ as a list of groups.  Let $q$ be the height of $P$.  With $I$ as in Lemma~\ref{lem:init-society}, we define $X_k = \{x\in P\st k\in I(x)\}$ when $1\le k \le q$ and set $S_0 = X_1,\ldots,X_q$.  This ordering is preserved throughout the evolution: if $Y$ appears before $Z$ in $S_0$ and $\{Y,Z\} \subseteq S_j$, then $Y$ also appears before $Z$ in $S_j$.  When $L$ is a list of objects $a_1, \ldots, a_n$, we define $\dist_L(a_i,a_j) = |j-i|$.  For convenience, when $Y$ and $Z$ are groups in $S_j$, we define $\dist_j(Y,Z) = \dist_{S_j}(Y,Z)$. 

Let $t=2(s-1)$.  In the initial $t$-society $(S_0, F_0)$, we define $F_0$ so that if $Y$ and $Z$ are distinct groups in $S_0$ with $\dist_0(Y,Z) \le s-1$, then $F(Y,k) = Z$ for some slot $k$.  If fewer than $2(s-1)$ groups in $S_0$ are at distance at most $s-1$ from $Y$, then some slots are empty (formally, $F(Y,k) = \star$).  Our replacement scheme maintains that in $t$-society $(S_j,F_j)$, a group $Y$ lists as friends all other groups $Z$ such that $\dist_j(Y,Z) \le s-1$.  This is possible to maintain since $\dist_j(Y,Z) < \dist_{j-1}(Y,Z)$ only occurs when some group $Z' \in S_{j-1}$ with $\dist_{j-1}(Y,Z') < \dist_{j-1}(Y,Z)$ does not survive the transition from $(S_{j-1},F_{j-1})$ to $(S_j,F_j)$.  It follows that at least as many of $Y$'s friendship slots become available as are needed to accommodate the groups $Z$ with $\dist_{j-1}(Y,Z)>s-1$ and $\dist_{j}(Y,Z)\le s-1$.  Our replacement scheme places these groups in $Y$'s available friendship slots arbitrarily.  As before, unused slots are assigned the value $\star$.  

\comment{
   First, for $Y \in S_{0}$, the friends of $Y$ in $(S_{0}, F_{0})$ are the groups $Z \in S_{0}$ distinct from $Y$ such that $\dist_0(Y,Z) \le s-1$. (These groups are arbitrarily assigned to the $t$ slots of $Y$; slots that are unused are formally given the special value $\star$.)
Now, let $j\geq 1$ and let us consider the transition between the $t$-societies $(S_{j-1},F_{j-1})$ and $(S_j,F_j)$. Let $Y \in S_{j}$ and denote by $D^{-}$ be the set of friends of $Y$ in $S_{j-1}$ that must be replaced, that is, 
$$
D^{-} = \{Z \in S_{j-1}: Z \notin S_{j} \text{ and } F_{j-1}(Y, k)= Z \text{ for some } k\in [t]\}.
$$
Let $D^{+}$ be the set of groups $Z \in S_{j}$ such that
$\dist_{j-1}(Y,Z) > s-1$ but $\dist_j(Y,Z) \le s-1$. Then $|D^{+}| \leq |D^{-}|$, since $t=2(s-1)$
and $\dist_{j-1}(Y,Z) \le s-1$ for every friend $Z$ of $Y$ in $(S_{j-1},F_{j-1})$. Choose an arbitrary injection $f: D^{+} \to D^{-}$. For every slot $F_{j}(Y, k)$ with $Z:=F_{j-1}(Y, k) \in D^{-}$, let $F_{j}(Y, k) := f^{-1}(Z)$ if $f^{-1}(Z) \neq \ns$ and $F_{j}(Y, k) := \star$ otherwise.
}

Our next aim is to show that our initial $t$-society and replacement scheme lead to a long evolution.  We first prove an analogue of Lemma~4.2 in~\cite{PRV_TALG}.

\begin{lemma}\label{lem:lem1-analogue}
Let $C_1, \ldots, C_m$ be a \ff , and define $C_j = \ns$ for $j>m$.  Let $(S_0,F_0)$ be our initial $t$-society, and let $(S_0,F_0),\ldots,(S_n,F_n)$ be the evolution resulting from our replacement scheme.  For each $i$, we have that $\bigcup_{X\in S_i} X \supseteq \bigcup_{j>i} C_j$.
\end{lemma}
\begin{proof}
By induction on $i$.  By Lemma~\ref{lem:init-society}, $I(x) \ne \ns$ for each element $x$, and therefore $\bigcup_{X\in S_0} X$ contains all elements in $P$.  

Let $i\ge 1$ and consider an element $y\in C_j$ with $j>i$.  Because $C_1,\ldots,C_m$ is a \ff , there is an element $z\in C_i$ such that $y$ and $z$ are incomparable.  By induction, there are groups $Y\in S_{i-1}$ and $Z\in S_{i-1}$ with $y\in Y$ and $z\in Z$.  Among all such pairs $\{Y,Z\}$, choose $Y$ and $Z$ to minimize $\dist_{i-1}(Y,Z)$.  We claim that $\dist_{i-1}(Y,Z)\le s-1$.  Indeed, if $\dist_{i-1}(Y,Z)\ge s$, then there are at least $s-1$ groups in $S_{i-1}$ that are strictly between $Y$ and $Z$ in the list $X_1, \ldots, X_q$.  By our selection of $Y$ and $Z$, none of these groups contain $y$ or $z$.  Hence, it follows that the index of each such group is strictly between $I(y)$ and $I(z)$, contradicting Lemma~\ref{lem:init-society}.

Because $\dist_{i-1}(Y,Z) \le s-1$, our replacement scheme ensures that $Y$ lists $Z$ as a friend in some slot.  Because $z \in Z \cap C_i$, some friend of $Y$ in $(S_{i-1},F_{i-1})$ has non-empty intersection with $C_i$.  It follows that $Y$ either makes an $\alpha$-transition or a $\beta$-transition from $S_{i-1}$ to $S_i$.  Hence $y\in Y \in S_i$ and therefore $y\in \bigcup_{X\in S_i}X$ as required.
\end{proof}

\begin{lemma}\label{lem:long-evolution}
Let $C_1, \ldots, C_m$ be a \ff , and define $C_j = \ns$ for $j>m$.  Let $(S_0,F_0)$ be our initial $t$-society, and let $(S_0,F_0),\ldots,(S_n,F_n)$ be the evolution resulting from our replacement scheme.  We have that $n\ge m+2$.
\end{lemma}
\begin{proof}
Let $y\in C_m$.  By Lemma~\ref{lem:lem1-analogue}, there is a group $Y\in S_{m-1}$ with $y\in Y$.  Because $Y$ has non-empty intersection with $C_m$, we have that $Y$ makes an $\alpha$-transition from $S_{m-1}$ to $S_m$.  Also, $N_{m-1,m}^\alpha(Y) = 1$ and $\eps((m+1)-(m-1)) = 2\eps = 1/t = 1/(2(s-1)) \le 1/2$, and therefore $Y$ is eligible to make a $\gamma$-transition from $S_m$ to $S_{m+1}$.  Hence $Y\in S_{m+1}$.  Because the evolution ends with an empty $t$-society, it follows that $n\ge m+2$.
\end{proof}

Putting all the pieces together, we obtain our main theorem.

\begin{theorem}
If $r$ and $s$ are at least $2$ and $P$ is an \rsfree\ poset of width $w$, then First-Fit partitions $P$ into at most $8(r-1)(s-1)w$ chains.
\end{theorem}
\begin{proof}
Let $C_1,\ldots,C_m$ be a \ff , and define $C_j = \ns$ for $j>m$.  Obtain our initial $t$-society $(S_0,F_0)$ from Lemma~\ref{lem:init-society}, and let $(S_0,F_0),\ldots,(S_n,F_n)$ be the evolution obtained with our replacement scheme.  By Lemma~\ref{lem:long-evolution}, we have that $n\ge m+2$.  By Lemma~\ref{lem:large-group}, some group $X\in S_0$ has size at least $(n-2)/4t = (n-2)/(8(s-1)) \ge m/(8(s-1))$.  By Lemma~\ref{lem:init-society}, the height of $X$ is at most $r-1$.  It follows that $X$ is the union of $r-1$ antichains, and therefore $w \ge |X|/(r-1) \ge m/(8(s-1)(r-1))$.
\end{proof}

\section{Concluding Remarks}

The following related problem is open: 
for which posets $Q$ of width $2$ is there a function $f_{Q}(w)$
such that First-Fit partitions every $Q$-free poset of width $w$ into at most $f_{Q}(w)$ chains?
The same question applies when $f_{Q}(w)$ is restricted to be a polynomial or a linear function of $w$.
We note that these problems are only interesting for posets $Q$ of width $2$. Indeed, there is
a trivial linear bound when $Q$ is a chain, and the example of Kierstead~\cite{K_CM} implies that no such function exists when the width of $Q$ is at least $3$.

\section*{Addendum}
While this article was under review, Bosek, Krawczyk, and Matecki~\cite{BKM_unpub} proved that for each poset $Q$ of width $2$, there is a function $f_{Q}(w)$ such that First-Fit partitions every $Q$-free poset of width $w$ into at most $f_{Q}(w)$ chains.  Our second question remains open. 

\section*{Acknowledgements}

This research was initiated while the authors were attending the First Montreal Spring School in
Graph Theory, held in Montreal in May 2010. The authors are grateful to the organizers of the school 
for providing a stimulating working environment. The second author thanks William T. Trotter for engaging talks on the subject and for relaying the problem studied in this article.

\bibliographystyle{plain}
\bibliography{paper}

\end{document}

%% file: commands.tex
\theoremstyle{plain}
\newtheorem{theorem}{Theorem}[section]
\newtheorem{lemma}[theorem]{Lemma}

\theoremstyle{definition}

\newcommand{\DEF}{\sl}       

\newcommand{\eps}{\varepsilon}  
\newcommand{\mc}{\mathcal}

\newcommand{\ns}{\varnothing}    

\newcommand{\chain}[1]{\underline{#1}}
\newcommand{\rsfree}{$(\chain{r}+\chain{s})$-free}
\def\st{\colon\,}    
\newcommand{\comment}[1]{}
\newcommand{\dist}{\mathrm{dist}}